\documentclass{amsart}
\usepackage[utf8]{inputenc}
\usepackage[english]{babel}
\usepackage{amsmath}
\usepackage{amsthm}
\usepackage{amssymb}
\usepackage{color}
\usepackage{dsfont}
\usepackage{hyperref}
\hypersetup{colorlinks=true,urlcolor=blue,citecolor=blue,linkcolor=blue}
\usepackage{tikz-cd}
\usepackage{colonequals}

\theoremstyle{plain}
\newtheorem{theorem}{Theorem}[section]
\newtheorem{lemma}[theorem]{Lemma}
\newtheorem{corollary}[theorem]{Corollary}
\newtheorem{propn}[theorem]{Proposition}

\newtheorem{remark}[theorem]{Remark}
\theoremstyle{definition}
\newtheorem{defn}[theorem]{Definition}
\newtheorem{eg}[theorem]{Example}

\title{Symmetric bilinear forms, superalgebras and integer matrix factorization}
\author{Dan Fretwell, Jenny Roberts.}
\date{}


\begin{document}

\begin{abstract}
   We construct and investigate certain (unbalanced) superalgebra structures on $\text{End}_K(V)$, with $K$ a field of characteristic $0$ and $V$ a finite dimensional $K$-vector space (of dimension $n\geq 2$). These structures are induced by a choice of non-degenerate symmetric bilinear form $B$ on $V$ and a choice of non-zero base vector $w\in V$. After exploring the construction further, we apply our results to certain questions concerning integer matrix factorization and isometry of integral lattices.
\end{abstract}

\maketitle

\section{Introduction}

Let $K$ be a field and $A$ be a $K$-algebra. We say that $A$ is a superalgebra over $K$ if it is equipped with a $\mathbb{Z}/2\mathbb{Z}$-grading, i.e.\ a vector space decomposition $A = A^{(0)}\oplus A^{(1)}$ satisfying $A^{(i)}A^{(j)} \subseteq A^{(i+j \bmod 2)}$ for all $i,j\in\{0,1\}$. Every $K$-algebra gives a trivial superalgebra structure by setting $A^{(0)} = A$ and $A^{(1)} = \{0\}$, but non-trivial examples of superalgebras exist, e.g.\ exterior algebras, Clifford algebras and polynomial rings (decomposition given by symmetric and antisymmetric polynomials). Superalgebras appear throughout Mathematics but are also of direct importance to Physics, since they provide robust mathematical models for supersymmetry (e.g.\ see \cite{deligne} and \cite{varadarajan}).

Let $K$ be a field and $V$ be a $K$-vector space of dimension $n\geq 2$. The endomorphism algebra $\text{End}_K(V)$ is then a $K$-algebra, and one asks whether there are any interesting superalgebra structures on this algebra. One particular well known example comes to mind. If a basis for $V$ is fixed then $\text{End}_K(V)\cong M_n(K)$, and a superalgebra structure is found by letting $A^{(0)}$ and 
$A^{(1)}$ be the subspaces of symmetric and antisymmetric matrices respectively. It is natural to wonder, as we have, whether there are other interesting superalgebra structures to be found beyond this example.

In this paper we provide a new class of superalgebra structures on $\text{End}_K(V)$ for any field $K$ of characteristic zero and any finite dimensional $K$-vector space $V$ (of dimension $n\geq 2$). These are induced from a choice of non-degenerate symmetric bilinear form $B$ on $V$ and a fixed choice of non-zero vector $w\in V$. More precisely, we show in Theorem \ref{thm:superalg} that there is a superalgebra decomposition: \[\text{End}_K(V) = E^{(0)}(B,w)\oplus E^{(1)}(B,w),\] where \begin{align*}E^{(0)}(B,w) &= \{ \phi \in \text{End}_K(V) \,|\, B(u, \phi(w)) = B(w, \phi(u)) = 0, \forall \ u \in \{w\}_B^\perp\},\\ E^{(1)}(B,w) &= \{ \phi \in \text{End}_K(V) \,|\, B(u, \phi(v)) =0, \forall \ u,v \in \{w\}_B^\perp, B(w, \phi(w)) =0 \}.\end{align*}

Our recipe is inspired by results coming from the papers \cite{higham}, \cite{hill} and \cite{lettington}, which correspond to the special case with $B$ the standard symmetric bilinear form and $w = (1,1,...,1)$ (after choosing a basis for $V$). In this case we have that $E^{(0)}(B,w)$ is the subspace of semi-magic squares and $E^{(1)}(B,w)$ is the subspace of matrices satisfying a vertex cross-sum property. We give a similar interpretation of the subspaces $E^{(i)}(B,w)$ in more generality, although we have tried to present our results in a basis independent way where possible.

Unlike the superalgebra structures given by Clifford algebras, the superalgebra structure given above is in general unbalanced. If $\text{dim}(V) = n$ then we have that $\text{dim}(E^{(0)}(B,w)) = n^2-2n+2$ and $\text{dim}(E^{(1)}(B,w)) = 2n-2$ (which are not equal if $n\geq 3$). Also, in general the above superalgebra structure is not equivalent to a symmetric and antisymmetric matrix decomposition (e.g.\ if $n\neq 4$ then we cannot have $(n^2-2n+2, 2n-2) = \left(\frac{n(n+1)}{2},\frac{n(n-1)}{2}\right)$).

After establishing the above result, we study potential implications towards the isometry problem for quadratic spaces. If $B'$ is another non-degenerate symmetric bilinear form on $V$ then $(V,B)$ and $(V,B')$ are isometric if $B' = B_{\phi}$ for some $\phi\in\text{GL}(V)$ (here $B_{\phi}(x,y) = B(\phi(x),\phi(y))$). Expanding $\phi$ into its ``odd" and ``even" parts according to the superalgebra structure induces a decomposition of $B_{\phi}$ into a sum of symmetric bilinear forms (Proposition \ref{propn:decomp}). Careful substitution then provides us with a set of necessary equations that must be satisfied by $\phi$ in order for the equality $B' = B_{\phi}$ to hold (Theorem \ref{thm:wt}). In practice, these equations can either be used to prove the non-existence of $\phi$ or to provide information about it.

Certain long standing problems in number theory concern the topic of integer matrix factorization. One particular family of such problems considers factorizations of the form $\mathcal{B}' = M^TM$ with $\mathcal{B}',M\in\text{GL}_n(\mathbb{Z})$ and $\mathcal{B}'$ symmetric of determinant $1$. Indeed, classical results (e.g.\ discussed in the paper of Mordell \cite{mordell}) show that if $n\leq 7$ and $\mathcal{B}'$ is positive definite then such a factorization is always guaranteed to exist (for $n=8$ a counterexample is given by any Gram matrix of the $E_8$ lattice under the standard Euclidean inner product).

The above problem can be instead viewed as an isometry problem for integral lattices. Interpreting $\mathcal{B}'$ as the Gram matrix of a non-degenerate symmetric bilinear form $B'$ on $V = \mathbb{Q}^n$, we see that the lattice $\Lambda = \mathbb{Z}^n$ is integral with respect to $B'$. The existence of $M$ in the matrix factorization is then equivalent to the equality $B' = B_{\phi}$ for some $\phi\in\text{GL}(\Lambda)\subset\text{GL}(V)$, where $B$ is the standard symmetric bilinear form (corresponding to Gram matrix $\mathcal{B} = I_n$ with respect to the standard basis). It is then clear that a special case of the above decomposition results (or at least their integral counterparts) can then be of use in first determining whether such an isometry $\phi$ exists and then determining precise information about it. This strategy was adopted by the authors in the papers \cite{higham} and \cite{lettington} to produce new insights on this problem (although a matrix superalgebra decomposition was used, as opposed to the analogous isometry decomposition that we consider). In particular, the authors use a superalgebra decomposition of $M_4(\mathbb{Q})$ to classify integer (and rational) matrix factorizations of the Wilson matrix: \[W = \begin{pmatrix}5 & 7 & 6 & 5\\ 7 & 10 & 8 & 7\\ 6 & 8 & 10 & 9\\ 5 & 7 & 9 & 10\end{pmatrix},\] a symmetric integral positive definite matrix of determinant $1$ (that is mildly ill-conditioned).

More generally, one could ask for integer matrix factorizations of the form $\mathcal{B}' = M^T\mathcal{B}M$ for an arbitrary $\mathcal{B},\mathcal{B}',M\in\text{GL}_n(\mathbb{Z})$ with $\mathcal{B},\mathcal{B}'$ symmetric and of equal determinant. As above, our more general results let us consider this question from the point of view of decompositions of lattice isometries (again, we require $B' = B_{\phi}$ for some $\phi\in \text{GL}(\Lambda)$, but now $B$ is not necessarily the standard symmetric bilinear form). We give a precise set of integral equations that must be satisfied for such an isometry $\phi$ to exist, and demonstrate via examples how these equations can be used to prove that certain integer matrix factorizations are impossible (alternatively proving that certain families of integral symmetric bilinear forms are non-isometric).

To summarise, our general approach and results lead to the following insights and improvements to the existing literature:

\begin{itemize}
\item{By considering an arbitrary non-degenerate symmetric bilinear form we are able to describe a whole family of superalgebra structures. The superalgebras constructed in \cite{higham}, \cite{hill}, \cite{lettington} are a special case of our construction (corresponding to the standard symmetric bilinear form and $w = (1,1,...,1)$).}
\item{By considering this family of superalgebra structures, we are able to provide results concerning the existence/non-existence of matrix factorizations of the form $\mathcal{B}' = M^T\mathcal{B}M$ with $M,\mathcal{B},\mathcal{B}'\in\text{GL}_n(\mathbb{Z})$ and $\mathcal{B},\mathcal{B}'$ symmetric. In comparison, the results in \cite{higham}, \cite{hill}, \cite{lettington} are restricted to the special case of $\mathcal{B} = I_n$. Since integer matrix factorizations of the form $\mathcal{B}' = M^TM$ are not always guaranteed to exist, it is beneficial to consider the more general problem mentioned above.}

\item{Our construction is coordinate free, which makes the results smoother and clearer to present. This approach has also provided new insights and strengthened previous results in this area. For example, Theorem $3.1$ of \cite{higham} provides a quadratic form obstruction to the existence of integer matrix factorizations of the form $\mathcal{B}' = M^TM$. This obstruction is equivalent to the first equation in our Corollary \ref{cor:quadform} (when specialised to the special case of $\mathcal{B}' = I_n$). The other equations in our corollary give the extra ``missing equations" necessary to complete the story.}

\item{We allow ourselves to take arbitrary base vector $w$ in our constructions (the papers \cite{higham}, \cite{hill}, \cite{lettington} only consider the choice $w = (1,1,...,1)$). This also helps with the above application, since other choices of $w$ can often lead to simpler sets of equations in Corollary \ref{cor:quadform}.}
\end{itemize}

\section{Symmetric bilinear forms and superalgebras}

Let $K$ be a field with $\text{char}(K) = 0$ and let $V$ be a finite dimensional $K$-vector space equipped with a non-degenerate symmetric bilinear form $B: V\times V \rightarrow K$. We assume that $\text{dim}_K(V) = n\geq 2$ and we fix a choice of vector $w\in V\backslash\{0\}$. Denote by $\{w\}_B^{\perp}$ the orthogonal complement of $w$ with respect to $B$. 

The aim of this section is to use the above data to obtain a superalgebra structure on $\text{End}_K(V)$, this being a decomposition of $K$-vector spaces \[\text{End}_K(V) = E^{(0)}(B,w)\oplus E^{(1)}(B,w)\] satisfying $E^{(i)}(B,w)E^{(j)}(B,w)\subseteq E^{(i+j \bmod 2)}(B,w)$ for all $i,j\in\{0,1\}$ (i.e.\ a  $\mathbb{Z}/2\mathbb{Z}$-graded algebra structure).

\subsection{The subspace $E^{(0)}(V)$}

We define the following subspace of $\text{End}_K(V)$:

\begin{equation*}
    E^{(0)}(B,w) = \{ \phi \in \text{End}_K(V) \,|\, B(u, \phi(w)) = B(w, \phi(u)) = 0, \forall \ u \in \{w\}_B^\perp\}.
\end{equation*}

\begin{lemma}
\label{lemma:equivSn}
$\phi \in E^{(0)}(B,w)$ if and only if there exists $\lambda\in K$ such that
\begin{align*}
    B(w,\phi(\cdot))&=\lambda B(w,\cdot), \\
    B(\cdot,\phi(w))&=\lambda B(\cdot,w).
\end{align*}
\end{lemma}

\begin{proof}
The reverse implication is clear. To prove the forward implication we first note that since $B(u,\phi(w)) =0$ for all $u \in \{w\}_B^\perp$, we have that $\phi(w) \in \{w\}_B^{\perp \perp}=Kw$, so that
\begin{equation*}
    \phi(w)=\lambda w
\end{equation*}
for some $\lambda \in K$. From here, we see that 
\begin{equation*}
    B(\cdot,\phi(w)) = \lambda B(\cdot,w).
\end{equation*}
For the second equation, let $\phi^\dagger$ be the adjoint of $\phi$ with respect to $B$. Then we have that $\phi^\dagger\in E^{(0)}(B,w)$, since: \begin{align*}B(u,\phi^\dagger(w)) &= B(\phi^\dagger(w),u) = B(w,\phi(u)) = 0\\ B(w,\phi^\dagger(u)) &= B(\phi^\dagger(u),w) = B(u,\phi(w)) = 0\end{align*} for all $u\in\{w\}_B^{\perp}$. It follows that: \[B(\cdot, \phi(w)) = B(\phi(w), \cdot) = B(w, \phi^\dagger(\cdot)) = \lambda' B(w,\cdot) = \lambda' B(\cdot, w),\] for some $\lambda'\in K$.

Finally, the fact that $\lambda = \lambda'$ follows by evaluating both maps at $w$: \[\lambda B(w,w) = B(w,\phi(w)) = B(\phi(w),w) = B(w,\phi^\dagger(w)) = \lambda' B(w,w),\] and noting that $B(w,w)\neq 0$ by non-degeneracy.
\end{proof}

Choosing a basis $\{e_1, e_2, ..., e_n\}$ for $V$, we get a matrix $M\in M_n(K)$ associated to $\phi$ and a Gram matrix $\mathcal{B} = (B(e_i,e_j))$ associated to $B$. The conditions above then mean that all row/column sums of the matrix $\mathcal{B}M$ (weighted with respect to the coordinates of $w$) are a fixed scaling of those of the matrix $\mathcal{B}$. In the particular case of $w = (1,1,...,1)$ and $\mathcal{B}=I_n$, we recover the conditions for $M$ to be a semi-magic square (as observed in \cite{higham} and \cite{lettington}).

\begin{defn}
If $\phi\in E^{(0)}(B,w)$ then we define the weight of $\phi$ as follows: \[\text{wt}(\phi)_{B,w} = \frac{B(w,\phi(w))}{B(w,w)}\]
(i.e.\ the constant $\lambda\in K$ corresponding to $\phi\in E^{(0)}(B,w)$).
\end{defn} 

In \cite{higham} and \cite{lettington}, a similar notion of weight was given to matrices $M\in M_n(K)$:
\begin{equation*}
    \text{wt}(M)=\frac{1}{n^2}\sum_i\sum_j m_{ij}=\frac{1}{n^2} 1_n^T M1_n.
\end{equation*}
where $1_n = (1,1,...,1)\in K^n$. By comparison, when $B$ is the standard symmetric bilinear form and $w = (1,1,...,1)$ (after choosing a basis for $V$) then our definition gives $\text{wt}(M)_{B,w} = n\,\text{wt}(M)$, a simple scaling of the above definition.

We can use the weight map to decompose the subspace $E^{(0)}(B,w)$ further. For $u,v\in V$ let $\phi_{B,u,v}\in\text{End}_K(V)$ be given by \[\phi_{B,u,v}(x) = B(v,x)u,\] the endomorphism corresponding to the outer product of $u$ and $v$ with respect to $B$ (i.e. in coordinates this is given by the matrix $uv^T \mathcal{B}$).

The following Lemma will be useful later.

\begin{lemma}
There exists a decomposition of $K$-vector spaces: \[E^{(0)}(B,w) = E^{(0)}_0(B,w)\oplus K\phi_{B,w,w},\] where $E^{(0)}_0(B,w)$ is the subspace consisting of weight $0$ endomorphisms.
\end{lemma}
\begin{proof}
Note first that the weight is a linear map $\text{End}_K(V) \rightarrow K$. The decomposition will then follow from the fact that $\phi_{B,w,w} \in E^{(0)}(B,w)$ has weight $B(w,w) \neq 0$. This is a simple consequence of Lemma \ref{lemma:equivSn}, since:

\begin{align*}
B(w, \phi_{B,w,w}(\cdot)) &= B(w,B(w,\cdot)w) = B(w,w)B(w,\cdot),\\
B(\cdot, \phi_{B,w,w}(w)) &= B(\cdot, B(w,w)w) = B(w,w)B(\cdot,w).
\end{align*}
\end{proof}

The above implies that every $\phi\in E^{(0)}(B,w)$ can be uniquely written in the form \[\phi = \phi_0 + \frac{\text{wt}(\phi)_{B,w}}{B(w,w)}\phi_{B,w,w}\] for some $\phi_0$ of weight $0$.

\subsection{The subspace $E^{(1)}(B,w)$}
\hfill\\

We define another subspace of $\text{End}_K(V)$: 
\begin{equation*}
    E^{(1)}(B,w) = \{ \phi \in \text{End}_K(V) \,|\, B(u, \phi(v)) =0, \forall \ u,v \in \{w\}_B^\perp, B(w, \phi(w)) =0 \}.
\end{equation*}

\begin{lemma}
\label{lemma:equivVn}
$\phi \in E^{(1)}(B,w)$ if and only if $\phi =\phi_{B,w,a}+\phi_{B,b,w}$ for some $a,b \in \{w\}_B^\perp$.
\end{lemma}

\begin{proof}
First note that if $\phi\in\text{End}_K(V)$ is of the above form then $\phi\in E^{(1)}(B,w)$, since for all $u,v\in\{w\}_B^{\perp}$ we have
\begin{align*}
    B(u, \phi(v)) &= B(u, \phi_{B,w,a}(v)) + B(u,\phi_{B,b,w}(v)) = 0, \\
    B(w, \phi(w)) &= B(w, \phi_{B,w,a}(w)) + B(w,\phi_{B,b,w}(w)) = 0.
\end{align*}

For the other direction, let $\phi \in E^{(1)}(B,w)$. Then for $v \in \{w\}_B^\perp$, we have that $\phi(v) \in \{w\}_B^{\perp \perp}=Kw$. So, there exists a linear map $f: \{w\}_B^\perp \rightarrow K$ such that $\phi(v)=f(v)w$ for all $v \in \{w\}_B^\perp$. Now, by the Riesz representation theorem there exists a vector $a \in \{w\}_B^\perp$ satisfying $f(v)=B(a,v)$ for all $v \in \{w\}_B^\perp$.  Thus $\phi(v) = \phi_{B,w,a}(v)$ for all $v\in\{w\}_B^{\perp}$.

To provide the necessary extension to $V$, let $b = \frac{1}{B(w,w)}\phi(w)\in V$ (well defined by non-degeneracy of $B$). Then $b \in \{w\}_B^\perp$, since $\phi \in E^{(1)}(B,w)$ and so
\begin{equation*}
    B(w,b) = \frac{1}{B(w,w)}B(w,\phi(w)) = 0.
\end{equation*}
It follows that \[\phi(w) = B(w,w)b = \phi_{B,b,w}(w),\] so that $\phi = \phi_{B,w,a}+\phi_{B,b,w}$.

\end{proof}

\subsection{The endomorphism superalgebra induced by $B$.}
\hfill\\

We are now ready to show that there is an orthogonal decomposition $\text{End}_K(V) = E^{(0)}(B,w)\oplus E^{(1)}(B,w)$, and that this defines a superalgebra structure on $\text{End}_K(V)$.

\begin{defn}
\label{def:bfend}
For $\phi_1, \phi_2\in\text{End}_K(V)$ we define: \[B(\phi_1, \phi_2) = \text{Tr}(\phi_1^\dagger\phi_2),\] where $\phi_1^{\dagger}$ is the adjoint of $\phi_1$ with respect to $B$.
\end{defn}

We check that this is a symmetric bilinear form. Bilinearity follows immediately since the adjoint and trace maps are $K$-linear. Symmetry follows since:
\begin{equation*}
    B(\phi_1, \phi_2) = \text{Tr}(\phi_1^\dagger\phi_2) = \text{Tr}(\phi_2\phi_1^\dagger)=\text{Tr}((\phi_2^\dagger\phi_1)^\dagger)= \text{Tr}(\phi_2^\dagger\phi_1) =  B(\phi_2,\phi_1).
\end{equation*}

The above is an analogue of the Frobenius inner product, but induced by $B$ instead of the standard bilinear form. Indeed, in matrix form we have $B(M_1, M_2) = \text{Tr}(\mathcal{B}^{-1}M_1^T\mathcal{B}M_2)$, and if $K\subseteq\mathbb{R}$ and the quadratic form corresponding to $B$ is positive definite then this bilinear form can be shown to be positive definite.

\begin{lemma}
\label{lemma:rewriteSV}
Let $i\in\{0,1\}$. If $\phi\in E^{(i)}(B,w)$ then $\phi^\dagger \in E^{(i)}(B,w)$. 
\end{lemma}

\begin{proof}
We already saw the case $i=0$ in the proof of Lemma \ref{lemma:equivSn}. For the case $i=1$, write $\phi = \phi_{B,w,a}+\phi_{B,b,w}$ for some $a,b\in\{w\}_B^{\perp}$. 

Note that $\phi_{B,w,a}^{\dagger} = \phi_{B,a,w}$ since: \[B(\phi_{B,w,a}(x),y) = B(B(a,x)w,y) = B(a,x)B(w,y) = B(B(w,y)a,x) = B(\phi_{B,a,w}(y),x)\] for all $x,y\in V$. Similarly $\phi_{B,b,w}^{\dagger} = \phi_{B,w,b}$ and so it follows that: \[\phi^\dagger = \phi_{B,w,a}^\dagger+\phi_{B,b,w}^\dagger = \phi_{B,a,w}+\phi_{B,w,b},\] and we notice that $a,b$ have simply switched roles. Thus $\phi^\dagger\in E^{(1)}(B,w)$.
\end{proof}

\begin{lemma}
\label{lem:decomp}
There exists an orthogonal decomposition \[\text{End}_K(V) = E^{(0)}(B,w)\oplus E^{(1)}(B,w)\] with respect to the symmetric bilinear form in Definition \ref{def:bfend}.
\end{lemma}

\begin{proof}
Let $\phi_1 \in E^{(0)}(B,w)$ and $\phi_2 \in E^{(1)}(B,w)$. Choose $a,b\in\{w\}_B^{\perp}$ such that $\phi_2 = \phi_{B,w,a}+\phi_{B,b,w}$. Then:
\begin{align*}
    B(\phi_1, \phi_2) = \text{Tr}(\phi_1^\dagger\phi_2)&= \text{Tr}(\phi_1^\dagger \phi_{B,w,a}) + \text{Tr}(\phi_1^\dagger \phi_{B,b,w}) \\
    &= B(w,\phi_1^\dagger(a)) + B(b,\phi_1^\dagger(w))\\
    &= 0
\end{align*}
since $\phi_1^\dagger\in E^{(0)}(B,w)$ (by Lemma \ref{lemma:rewriteSV}). It follows that $E^{(0)}(B,w)$ and $E^{(1)}(B,w)$ are orthogonal. 

Since $B$ is non-degenerate, $\text{dim}(E^{(0)}(B,w)) = n^2-2n+2$ and $\text{dim}(E^{(1)}(B,w)) = 2n-2$, we see that $E^{(0)}(B,w)^{\perp} = E^{(1)}(B,w)$ and so the decomposition follows.
\end{proof}

We are now in a position to prove the main result. We will do this in a way that also highlights certain extra features of the decomposition (e.g.\ multiplicativity of the weight).
\begin{theorem}
\label{thm:superalg}
The decomposition in Lemma \ref{lem:decomp} defines a superalgebra structure on $\text{End}_K(V)$, i.e. $E^{(i)}(B,w)E^{(j)}(B,w) \subseteq E^{(i+j \bmod 2)}(B,w)$ for all $i,j\in\{0,1\}$.
\end{theorem}

\begin{proof}
First let $\phi_1,\phi_2\in E^{(0)}(B,w)$ have weights $\lambda_1, \lambda_2\in K$ respectively. Then $\phi_1\phi_2\in E^{(0)}(B,w)$ has weight $\lambda_1\lambda_2$ since:
\begin{align*}
    B(\cdot. (\phi_1\phi_2)(w)) = B(\phi_1^\dagger(\cdot),\phi_2(w)) = \lambda_2 B(\phi_1^\dagger(\cdot), w) = \lambda_1\lambda_2 B(\cdot, w), \\
    B(w, (\phi_1\phi_2)(\cdot)) = B(\phi_1^\dagger(w),\phi_2(\cdot)) = \lambda_1 B(w,\phi_2(\cdot)) = \lambda_1\lambda_2 B(w,\cdot).
\end{align*}
Now let $\phi_3,\phi_4\in E^{(1)}(B,w)$ and suppose that $a_1, b_1, a_2, b_2\in \{w\}_B^{\perp}$ are such that \begin{align*}\phi_3 &= \phi_{B,w,a_1} + \phi_{B,b_1,w},\\ \phi_4 &= \phi_{B,w,a_2} + \phi_{B,b_2,w}.\end{align*} 

From this one checks that $(\phi_3\phi_4)(w) = (\phi_4^\dagger\phi_3^\dagger)(w) = B(a_1,b_2)B(w,w) w$, and so:
\begin{align*}
    B(\cdot, (\phi_3\phi_4)(w)) &= B(a_1,b_2)B(w,w) B(\cdot,w),\\
    B(w,(\phi_3\phi_4)(\cdot)) &= B((\phi_4^\dagger\phi_3^\dagger)(w),\cdot) = B(a_1,b_2)B(w,w)B(w,\cdot)
\end{align*}

Since $B(a_1,b_2)B(w,w)\neq 0$ we see that $\phi_3\phi_4\in E^{(0)}(B,w)$ has weight $B(a_1,b_2)B(w,w)$.

We now show that $\phi_1\phi_3$ and $\phi_3\phi_1$ lie in $E^{(1)}(B,w)$. First note that:
\begin{align*}
    B(w,(\phi_1\phi_3)(w)) &= \lambda_1 B(w,\phi_3(w)) = 0,\\
    B(w,(\phi_3\phi_1)(w)) &= B(\phi_3^\dagger(w),\phi_1(w)) = \lambda_1 B(\phi_3^\dagger(w),w) = 0.
\end{align*}

Finally, if $u,v\in\{w\}_B^{\perp}$, we have:
\begin{align*}
    B(u,(\phi_1\phi_3)(v)) &= B(a_1,v)B(u,\phi_1(w)) = \lambda_1 B(a_1,v)B(u,w) = 0, \\
    B(u,(\phi_3\phi_1)(v)) &= B((\phi_1^\dagger\phi_3^\dagger)(u),v) = B(b_1,u)B(\phi_1^\dagger(w),v) = \lambda_1 B(b_1,u)B(w,v) = 0.
\end{align*}
\end{proof}

Note that the even and odd pieces of this superalgebra decomposition have different dimensions if $n\geq 3$. This is in contrast with the Clifford algebra coming from the corresponding quadratic form on $V$. This decomposition is also not equivalent to ones given by symmetric and antisymmetric matrices in general (e.g.\ if $n\neq 4$ then we cannot have $(n^2-2n+2, 2n-2) = \left(\frac{n(n+1)}{2},\frac{n(n-1)}{2}\right)$).

We pause to give a simple example of the above superalgebra decomposition.

\begin{eg}
Let $n=3$ and $K = \mathbb{Q}$. Choose the standard basis for $V = \mathbb{Q}^3$. The following symmetric bilinear form on $V$ is non-degenerate: \[B(u,v) = u_1v_1 + 2u_2v_2 + 3u_3v_3.\]

For the choice $w = (1,0,0)$ we get the superalgebra decomposition $\text{End}_K(V)\cong M_3(\mathbb{Q}) = E^{(0)}(B,w)\oplus E^{(1)}(B,w)$ with: \begin{align*}E^{(0)}(B,w) &= \left\{\begin{pmatrix}a & 0 & 0\\ 0 & e & f\\ 0 & h & i\end{pmatrix}\,\Bigg|\, a,e,f,h,i\in\mathbb{Q}\right\},\\ E^{(1)}(B,w) &= \left\{\begin{pmatrix}0 & b & c\\ d & 0 & 0\\ g & 0 & 0\end{pmatrix}\,\Bigg|\, b,c,d,g\in\mathbb{Q}\right\}.\end{align*} However, for the choice $w=(1,1,1)$ we get the very different looking superalgebra decomposition with components: \begin{align*}E^{(0)}(B,w) &= \left\{\begin{pmatrix}a & b & c\\ d & e & a + b + c - d - e\\ \frac{b+c-2d}{3} & \frac{2a + b + 2c - 2e}{3} & \frac{a + b + 2d + 2e}{3}\end{pmatrix}\,\Bigg|\, a,b,c,d,e\in\mathbb{Q} \right\}, \\ E^{(1)}(B,w) &= \left\{\begin{pmatrix}f & g& h\\ i & -2f + g + 2i & -3f + h + 3i \\ \frac{4f - g - h - 2i}{3} & \frac{2f + g - 2h - 4i}{3} & f - g - 2i\end{pmatrix}\,\Bigg|\, f,g,h,i\in\mathbb{Q}\right\}.\end{align*}
\end{eg}

A natural question to ask in general is whether the above recipe produces isomorphic superalgebras as $(B,w)$ vary. The theorem below addresses this question. First recall that if $\phi_1, \phi_2\in\text{End}_K(V)$ then we can construct a new symmetric bilinear form on $V$ via \[B_{\phi_1,\phi_2}(x,y) = \frac{1}{2}(B(\phi_1(x),\phi_2(y))+B(\phi_1(y),\phi_2(x))).\] When $\phi_1=\phi_2=\phi$ we will simply write $B_{\phi}$.

\begin{theorem}
Suppose that $\phi\in\text{GL}(V)$. Then conjugation by $\phi$ gives a superalgebra isomorphism: \[E^{(0)}(B,w)\oplus E^{(1)}(B,w) \cong E^{(0)}(B_{\phi},\phi^{-1}(w))\oplus E^{(1)}(B_{\phi},\phi^{-1}(w)).\]
\end{theorem}

\begin{proof}
We first note that the map $\psi\mapsto \phi^{-1}\psi\phi$ clearly defines an algebra isomorphism $\text{End}_K(V)\rightarrow \text{End}_K(V)$. It then suffices to prove that this map preserves the grading.

Suppose that $\psi\in E^{(0)}(B,w)$. Then $\phi^{-1}\psi\phi\in E^{(0)}(B_{\phi},\phi^{-1}(w))$ since for any $u\in \{\phi^{-1}(w)\}^{\perp}_{B_{\phi}}$ we have (using Lemma \ref{lemma:equivSn}): \begin{align*}B_{\phi}(u, (\phi^{-1}\psi\phi)(\phi^{-1}(w))) &= B(\phi(u), \psi(w)) = \lambda B(\phi(u),w) = \lambda B_{\phi}(u,\phi^{-1}(w)) = 0,\\ B_{\phi}((\phi^{-1}\psi\phi)(u),\phi^{-1}(w)) &= B(\psi(\phi(u)), w) = \lambda B(\phi(u), w) = \lambda B_{\phi}(u, \phi^{-1}(w)) = 0.\end{align*}

Similarly, if $\psi\in E^{(1)}(B,w)$ then $\phi^{-1}\psi\phi\in E^{(1)}(B_{\phi},\phi^{-1}(w))$ since for any $u,v\in \{\phi^{-1}(w)\}^{\perp}_{B_{\phi}}$ we have: \begin{align*}B_{\phi}(u, (\phi^{-1}\psi\phi)(v)) &= B(\phi(u), \psi(\phi(v))) = 0,\\ B_{\phi}(\phi^{-1}(w), (\phi^{-1}\psi\phi)(\phi^{-1}(w)) &= B(w, \psi(w)) = 0\end{align*} (noting that $u,v\in \{\phi^{-1}(w)\}^{\perp}_{B_{\phi}}$ implies that $\phi(u),\phi(v)\in \{w\}^{\perp}_{B}$). 

\end{proof}

\section{Further results} 

In this section we study some simple consequences of the superalgebra decomposition. In particular, we will see how it can be used to deduce information about the corresponding isometry problem for the non-degenerate symmetric bilinear space $(V,B)$. 

To recall, the isometry problem for $(V,B)$ is the problem of detecting whether a given non-degenerate symmetric bilinear form $B'$ on $V$ can be written as $B' = B_{\phi}$ for some $\phi\in\text{GL}(V)$ (and further to construct such a map $\phi$ if it exists).

First, we consider the decomposition of $B_{\phi}$ for an arbitrary $\phi\in\text{End}_K(V)$, induced by the decomposition of $\phi$ given by Lemma \ref{lem:decomp}.

\begin{propn}
\label{propn:expansion}
Let $\phi\in\text{End}_K(V)$ have decomposition \[\phi = \phi_0 + \frac{\text{wt}(\phi)_{B,w}}{B(w,w)}
\phi_{B,w,w}+ \phi_{B,w,a}+\phi_{B,b,w},\] for some $a,b\in\{w\}_B^{\perp}$ and $\phi_0\in E^{(0)}(B,w)$ of weight $0$. Then \[B_{\phi} = B_{\phi_0} + \frac{\text{wt}(\phi)_{B,w}^2}{B(w,w)^2} B_{\phi_{B,w,w}} + B_{
\phi_{B,w,a}} + B_{\phi_{B,b,w}} + 2\left(B_{\phi_0,\phi_{B,b,w}} + \frac{\text{wt}(\phi)_{B,w}}{B(w,w)}B_{\phi_{B,w,w},\phi_{B,w,a}}\right).\]
\end{propn}

\begin{proof}
Writing $\phi$ as above and expanding gives \begin{align*}B_{\phi} &= B_{\phi_0} + \frac{\text{wt}(\phi)_{B,w}^2}{B(w,w)^2} B_{\phi_{B,w,w}} + B_{\phi_{B,w,a}} + B_{\phi_{B,b,w}}\\ &+2\frac{\text{wt}(\phi)_{B,w}}{B(w,w)}( B_{\phi_0,\phi_{B,w,w}} + B_{\phi_{B,w,w},\phi_{B,w,a}}+B_{\phi_{B,w,w},\phi_{B,b,w}})\\ &+ 2(B_{\phi_0,\phi_{B,w,a}} + B_{\phi_0,\phi_{B,b,w}} + B_{\phi_{B,w,a},\phi_{B,b,w}}).\end{align*}

\noindent The seventh and tenth terms vanish since $a,b\in\{w\}_B^{\perp}$, whereas the fifth and eighth terms vanish since $\phi_0\in E^{(0)}(B,w)$ has weight $0$. \end{proof}

The above Lemma allows us to either extract information about potential isometries $\phi$ satisfying the equality $B' = B_{\phi}$, or to give potential obstructions to their existence.

\begin{theorem}
\label{thm:wt}
Let $B'$ be an arbitrary non-degenerate symmetric bilinear form on $V$. If $B' = B_{\phi}$ for some $\phi\in\text{End}_K(V)$ then the following holds for any $w\in V\backslash\{0\}$ and $z_1, z_2\in\{w\}_B^{\perp}$:
\begin{align*}
    B'(w,w) &= B(w,w)(\text{wt}(\phi)_{B,w}^2 + B(w,w)B(b,b)),\\
    B'(w,z_1) &= B(w,w)(\text{wt}(\phi)_{B,w}B(a,z_1)+B(b,\phi_0(z_1))),\\
    B'(z_1,z_2) &= B(\phi_0(z_1), \phi_0(z_2)) + B(a,z_1)B(a,z_2)B(w,w).
\end{align*}
Here $a,b\in\{w\}_B^{\perp}$ are as in Proposition \ref{propn:expansion}.
\end{theorem}
\begin{proof}
If $B' = B_{\phi}$ then Proposition \ref{propn:expansion} tells us that: \[B' = B_{\phi_0} + \frac{\text{wt}(\phi)_{B,w}^2}{B(w,w)^2} B_{\phi_{B,w,w}} + B_{
\phi_{B,w,a}} + B_{\phi_{B,b,w}} + 2\left(B_{\phi_0,\phi_{B,b,w}} + \frac{\text{wt}(\phi)_{B,w}}{B(w,w)}B_{\phi_{B,w,w},\phi_{B,w,a}}\right).\]

The claim follows by evaluating both sides at each pair of vectors, recalling that $\phi_0$ has weight $0$ (so that $B(\phi_0(w),x) = B(w,\phi_0(x)) = 0$ for any $x\in V$) and that $a,b, z_1, z_2$ are orthogonal to $w$.
\end{proof}

In the special case of $B$ being the standard symmetric bilinear form and $w=(1,1,...,1)$, we find that equation one of Theorem \ref{thm:wt} is equivalent to equation $(18)$ of \cite{lettington} and equation $(3.1)$ of \cite{higham} (both of which proved useful in determining integer matrix factorizations). Theorem \ref{thm:wt} allows us to consider the case of arbitrary non-degenerate symmetric bilinear form, and fills in the ``missing equations" that must be solved to determine an isometry $\phi$ (if it exists).

\begin{remark}
    When $K\subseteq \mathbb{R}$ and $B'$ is an inner product then we can apply the Cauchy-Schwarz inequality to give following constraint for any $z\in\{w\}_B^{\perp}$:
    \[|B'(w,z)|^2 \leq B'(w,w)B'(z,z).\] This implies a corresponding inequality between the quantities on the RHS in Theorem \ref{thm:wt}.
\end{remark}

\section{Application to integer matrix factorization}

In this section we let $K = \mathbb{Q}$ and fix a lattice $\Lambda\subset V$ (i.e.\ a $\mathbb{Z}$-submodule satisfying $\mathbb{Q}\Lambda = V$). Further, we assume that $B(v,v')\in\mathbb{Z}$ for all $v,v'\in\Lambda$, so that $\Lambda$ is an integral lattice with respect to $B$. In particular, this implies that $\Lambda\subseteq\Lambda^* = \{v\in V\,|\, B(u,v)\subseteq\mathbb{Z}, \forall u\in\Lambda\}$.

We are interested in studying the isometry problem for $(\Lambda,B)$. This is the problem of detecting whether a given non-degenerate symmetric bilinear form $B'$ on $V$ can be written as $B' = B_{\phi}$ for some $\phi\in\text{GL}(\Lambda)=\{\phi\in\text{GL}(V)\,|\,\phi(\Lambda) = \Lambda\}$ (and further to construct such a map if it exists). Fixing a basis for $\Lambda$, this is equivalent to giving an integer matrix factorization of the form $\mathcal{B}' = M^T\mathcal{B}M$ for some $M\in\text{GL}_n(\mathbb{Z})$ (where $\mathcal{B},\mathcal{B}'$ are Gram matrices of $B, B'$ respectively, with respect to this basis for $\Lambda$). 

Such problems have been studied in number theory for a long time. For example, in the case $K = \mathbb{Q}$ and $\mathcal{B} = I_n$ (i.e.\ the standard symmetric bilinear form on $\Lambda = \mathbb{Z}^n \subset V = \mathbb{Q}^n$), it is a well known theorem (e.g.\ see the paper of Mordell \cite{mordell}) that such a factorization is always possible if $n\leq 7$, $\text{det}(\mathcal{B}') = 1$ and $\mathcal{B}'$ is positive definite. In other words, every symmetric positive definite matrix $\mathcal{B}'\in\text{SL}_n(\mathbb{Z})$ has a factorization $\mathcal{B}' = M^TM$ for some $M\in\text{GL}_n(\mathbb{Z})$ whenever $n\leq 7$. In dimension $n=8$ there may not be such a factorization. Indeed, any Gram matrix coming from a basis of the $E_8$ lattice (under the usual Euclidean inner product) provides a counter-example. As the rank $n$ grows, the number of counter-examples grows rapidly. 

As mentioned previously, in the papers \cite{higham} and \cite{lettington}, the authors utilised a superalgebra decomposition of $M_n(\mathbb{Q})$ in order to provide further understanding of this special case of the isometry problem. In doing so, they were able to provide a quadratic form obstruction towards the existence of such isometries (e.g.\ see Theorem $3.1$ of \cite{higham}). When such an isometry does exists, information about the isometry is learned from the corresponding integer points of the quadratic form. The results of the previous sections now allow us to extend the results in the above papers to study more general isometry problems through this lens. They will also strengthen the results of the above mentioned papers by filling in the ``missing equations" necessary to complete the story (see Corollary \ref{cor:quadform} below).

\begin{propn}
\label{propn:decomp}
If $w\in\Lambda$ and $\phi\in \mathrm{GL}(\Lambda)$ decomposes as \[\phi = \phi^{(0)} + \phi^{(1)} = \phi_0 + \frac{\text{wt}(\phi)_{B,w}}{B(w,w)}\phi_{B,w,w} + \phi_{B,w,a}+\phi_{B,b,w}\] with $a,b\in\{w\}_B^{\perp}$, then $a\in \frac{1}{B(w,w)^2}\Lambda^*$ and $b\in \frac{1}{B(w,w)^2}\Lambda$.
\end{propn}

\begin{proof}
We first claim that $\phi^{(0)}(w)\in\frac{1}{B(w,w)}\Lambda$ and $\phi^{(1)}(w)\in\frac{1}{B(w,w)}\Lambda$. The first inclusion follows from the proof of Lemma \ref{lemma:equivSn} since: \begin{align*}\phi^{(0)}(w) = \frac{B(w,\phi^{(0)}(w))}{B(w,w)}w =  \frac{B(w,\phi^{(0)}(w)+\phi^{(1)}(w))}{B(w,w)}w = \frac{B(w,\phi(w))}{B(w,w)}w \in \frac{1}{B(w,w)}\Lambda,\end{align*} and the second follows from the fact that $\phi^{(1)}(w) = \phi(w) - \phi^{(0)}(w)\in\frac{1}{B(w,w)}\Lambda$. Note that the proof of Lemma \ref{lemma:equivVn} gives $b = \frac{1}{B(w,w)}\phi^{(1)}(w)$, and so $b\in\frac{1}{B(w,w)^2}\Lambda$. For the claim about $a$ we recall that by Lemma \ref{lemma:rewriteSV} we have that
\begin{align*}
    \phi^\dagger = (\phi^\dagger)^{(0)} + (\phi^\dagger)^{(1)} = (\phi^{(0)})^{\dagger} + (\phi^{(1)})^{\dagger},
\end{align*} since $(\phi^{(0)})^{\dagger}\in E^{(0)}(B,w)$ and $(\phi^{(1)})^{\dagger}\in E^{(1)}(B,w)$.
Note that $\phi(\Lambda) = \Lambda$ implies that $\phi^{\dagger}(\Lambda^*) = \Lambda^*$. We also have $w\in\Lambda\subseteq\Lambda^*$, and so we may apply the same argument as above to conclude that $(\phi^{(0)})^{\dagger}(w)\in\frac{1}{B(w,w)}\Lambda\subseteq\frac{1}{B(w,w)}\Lambda^*$, hence that $(\phi^{(1)})^{\dagger} = \phi^{\dagger}(w) - (\phi^{(0)})^{\dagger}(w) \in\frac{1}{B(w,w)}\Lambda^*$. Finally, since the adjoint swaps the roles of $a$ and $b$ in the $E^{(1)}(B,w)$ component we learn that $a = \frac{1}{B(w,w)}(\phi^{(1)})^{\dagger}(w)\in\frac{1}{B(w,w)^2}\Lambda^*$.
\end{proof}

\begin{corollary}
\label{cor:quadform}
Suppose that $B'$ is a non-degenerate symmetric bilinear form on $V$. If $B' = B_{\phi}$ for some $\phi\in\text{GL}(\Lambda)$ then the vectors $\tilde{a} = B(w,w)^2a\in\Lambda^*$, $\tilde{b} = B(w,w)^2b\in\Lambda$ from Proposition \ref{propn:decomp} satisfy:

\begin{align*}
B(w,w)^2B'(w,w) &= B(w,w)B(w,\phi(w))^2 +B(\tilde{b},\tilde{b}),\\
B(w,w)^3B'(w,z) &= B(w,w)B(w,\phi(w))B(\tilde{a},z_0) + B(\tilde{b}, \tilde{\phi}_0(z_0)),\\
B(w,w)^4B'(z,z) &= B(\tilde{\phi}_0(z_0), \tilde{\phi}_0(z_0))+B(w,w)B(\tilde{a},z_0)^2,\\
\end{align*} 
for any $z_0\in\Lambda$ (with $z = z_0 - \frac{B(z_0,w)}{B(w,w)}w\in \{w\}^{\perp}$ and $\tilde{\phi}_0(z_0) = B(w,w)^2\phi_0(z_0)\in\Lambda$).
\end{corollary}

\begin{proof}
By Theorem \ref{thm:wt} we have that \[B'(w,w) = B(w,w)(\text{wt}(\phi)_{B,w}^2 + B(w,w)B(b,b)).\] Clearing denominators gives the first equation. Also note that
\begin{align*}
\phi(z) &=\phi_0(z) + B(a,z)w\\ &= \phi_0(z_0) + \frac{B(z_0,w)}{B(w,w)}\phi_0(w) + \left(B(a,z_0) + \frac{B(z_0,w)}{B(w,w)}B(a,w)\right)w\\ &= \phi_0(z_0) + B(a,z_0)w.
\end{align*}

Since $\phi(z)\in \frac{1}{B(w,w)}\Lambda$ and $B(a,z_0)w = \frac{1}{B(w,w)^2}B(\tilde{a},z_0)w\in \frac{1}{B(w,w)^2}\Lambda$, we find that $\phi_0(z_0)\in\frac{1}{B(w,w)^2}\Lambda$. Again, by Theorem \ref{thm:wt} we have: \begin{align*}B'(w,z_1) &= B(w,w)(\text{wt}(\phi)_{B,w}B(a,z_1)+B(b,\phi_0(z_1))),\\ B'(z_1,z_2) &= B(\phi_0(z_1), \phi_0(z_2)) + B(a,z_1)B(a,z_2)B(w,w).\end{align*} Inserting $z_1 = z_2 = z$  and clearing denominators gives the second and third equations.
\end{proof}

Corollary \ref{cor:quadform} tells us that information about potential lattice isometries can be determined by solving certain Diophantine equations. For example, the first and third equations define integral quadratic forms, and if the underlying quadratic space is positive definite then these quadratic forms are also positive definite (allowing for a finite search). If there are no solutions to any of these equations then we immediately deduce that no such isometry exists.

As mentioned above, after choosing a basis for $\Lambda$ we see that the above equations give insight into the solution of integer matrix factorization problems of the form $\mathcal{B}' = M^T \mathcal{B}M$ with $M,\mathcal{B},\mathcal{B}'\in \text{GL}_n(\mathbb{Z})$ and $\mathcal{B},\mathcal{B}'$ symmetric (Gram matrices).

In the examples below we consider the space $V = \mathbb{Q}^n$ and lattice $\Lambda = \mathbb{Z}^n$ (i.e.\ $\mathbb{Z}$-span of the standard basis). In  
Examples \ref{eg:rank2} to \ref{eg:rank3gen} we find that the first equation of Corollary \ref{cor:quadform} suffices to prove that no isometry exists. Examples \ref{eg:rank4} and \ref{eg:wilson} require us to use all three equations to obtain sufficient information to deduce whether or not there is a solution to the integer matrix factorisation. This highlights a clear difference with \cite{higham} where they were only able to use the first equation.

\begin{eg}
\label{eg:rank2}
If $n=2$ then \begin{align*}B(u,v) &= u_1v_1 + 5u_2v_2,\\ B'(u,v) &= 2u_1v_1 + u_1v_2 + u_2v_1 + 3u_2v_2,\end{align*} are non-degenerate symmetric bilinear forms on $V$, and we see that $\Lambda$ is integral with respect to both forms. 

Suppose that $B' = B_{\phi}$ for some $\phi\in\text{GL}(\Lambda) \cong \text{GL}_2(\mathbb{Z})$. Choosing $w = (1,0)$ gives $\tilde{b} = b = (0,t)\in\mathbb{Z}^2$, $B(w,w) = 1$ and $B'(w,w) = 2$. The first equation of Corollary \ref{cor:quadform} then implies that there are integer solutions to the equation: \[2 = B(w,\phi(w))^2 + 5t^2.\] This is a contradiction, hence $(\Lambda,B')$ and $(\Lambda,B)$ are not isometric (despite these forms having Gram matrices of the same determinant, a necessary condition for isometry). 
\end{eg}

We can also use Corollary \ref{cor:quadform} to prove that certain infinite families of pairs of forms in rank 2 and rank 3 lattices cannot be isometric, despite their corresponding Gram matrices having the same determinant.

\begin{eg}
\label{eg:rank2gen}
If $n=2$ then \[B=(u,v) = m^2u_1v_1 + n^2u_2v_2\] is a non-degenerate symmetric bilinear form on $V$ whenever $m,n\in\mathbb{Q}\backslash\{0\}$. Suppose that $m,n\in\mathbb{Z}\backslash\{0\}$. Then $\Lambda$ is an integral lattice with respect to $B$. 

For any choice of $\alpha,\beta,\gamma\in\mathbb{Z}$ such that $\alpha\gamma - \beta^2 = (mn)^2$, we have that \[B'(u,v) = \alpha u_1v_1 + \beta u_1v_2 + \beta u_2v_1 + \gamma u_2v_2\] is another non-degenerate symmetric bilinear form on $V$ with the same determinant as $B$. Also, $\Lambda$ is integral with respect to this form.

Suppose that there exists a prime $p\equiv 3 \bmod 4$ dividing $\alpha$ to an odd power. Then $B' = B_{\phi}$ for some $\phi\in\text{GL}(\Lambda)$ is impossible, since choosing $w = (1,0)$ in the first equation of Corollary \ref{cor:quadform} gives: \[\alpha\cdot m^4 = m^2B(w,\phi(w))^2 + n^2t^2 = (mB(w,\phi(w))^2 + (nt)^2.\] This has no integer solutions by Fermat's theorem on sums of two squares (e.g.\ Theorem $13.3$ of \cite{burton} and the Corollary that follows). It follows that $(\Lambda, B)$ and $(\Lambda,B')$ are not isometric.
\end{eg}

\begin{eg}
\label{eg:rank3gen}
If $n=3$ then \begin{align*}B(u,v) &= (2m^2+1)u_1v_1 - u_1v_2 - u_2v_1 + u_2v_2 + 2m^2u_3v_3,\\ B'(u,v) &= 4m^3u_1v_1 + mu_2v_2 + u_3v_3,\end{align*} are non-degenerate symmetric bilinear forms on $V$ for any $m\in\mathbb{Q}\backslash\{0\}$. Suppose that $m\in\mathbb{Z}\backslash\{0\}$. Then $\Lambda$ is an integral lattice with respect to $B$ and also with respect to $B'$. As in previous examples, these forms have the same determinant and so there is a chance that $(\Lambda,B)$ and $(\Lambda,B')$ are isometric.

Suppose that $B' = B_{\phi}$ for some $\phi\in\text{GL}(\Lambda)$. Choosing $w = (1,1,1)$ gives $\tilde{b} = (b_1, b_2, -b_1)$ for some $b_1,b_2\in\mathbb{Z}$, since $B(w,\tilde{b}) = 0$ if and only if $2m^2(b_1 + b_3) = 0$. It is then readily calculated that \begin{align*}B(w,w) &= 4m^2,\\ B'(w,w) &= (1+m+4m^3),\\ B(\tilde{b},\tilde{b}) &= (4m^2 + 1)b_1^2 - 2b_1b_2 + b_2^2 = (2mb_1)^2 + (b_1-b_2)^2,\end{align*} so that equation one of Corollary \ref{cor:quadform} gives: \begin{align*}16m^4(4m^3+m+1) &= 4m^2B(w,\phi(w))^2 + (2mb_1)^2 + (b_1-b_2)^2,\\ &= (2mB(w,\phi(w))^2 + (2mb_1)^2 + (b_1-b_2)^2.\end{align*}

By Legendre's theorem on sums of three squares (e.g.\ Theorem $13.5$ of \cite{burton}), there are no integer solutions to this equation whenever $16m^4(4m^3+m+1)$ is of the form $4^t(8k+7)$ for some $t,k\geq 0$. It follows that in this case there is no such isometry between $(\Lambda,B)$ and $(\Lambda,B')$.

Since any odd integer satisfies $m^4\equiv 1 \bmod 8$, the above condition is equivalent to $4m^3 + m  + 1 = 4^t(8k+7)$ for some $t,k\geq 0$. By a tedious case by case check, one finds that asymptotically this happens for $\frac{1}{6}$ of the possible $m$ values. Thus our method has proved that at least $\frac{1}{6}$ of the pairs of forms in this family cannot be isometric over $\mathbb{Z}$. In fact, there are no values of $m$ for which these pairs of forms are isometric, and this can be seen by similar calculations for other choices of vector $w$.

Similarly, for any choice of $\alpha,\beta,\gamma\in\mathbb{Z}\backslash\{0\}$ satisfying $\alpha\gamma - \beta^2 = 4m^4$ the form: \[B''(u,v) = \alpha u_1v_1 + \beta u_1v_2 + \beta u_2v_1 + \gamma u_2v_2 + u_3v_3\] is a non-degenerate symmetric bilinear form on $V$ such that $\Lambda$ is integral with respect to $B''$ (and also has the same determinant as $B$).

Suppose again that $w=(1,1,1)$ and that $B''(w,w) = \alpha + 2\beta + \gamma + 1 = 4^t(8k+7)$ for some $t,k\geq 0$. Then the same argument as above shows that $(\Lambda,B)$ and $(\Lambda, B'')$ are not isometric either, by Legendre's theorem on sums of three squares.

Numerical examples of the above are given by $m = 3$ and $(\alpha,\beta,\gamma) = (60,6,6)$. Since we have that \begin{align*}4m^3+m+1 &= 28 = 4\cdot 7,\\ \alpha\gamma-\beta^2 &= 4m^4 = 324,\\ \alpha+2\beta+\gamma+1 &= 79 \equiv 7 \bmod 8,\end{align*} we conclude that the following forms are such that $(\Lambda,B)$ and $(\Lambda,B')$ are not isometric and $(\Lambda,B)$ and $(\Lambda,B'')$ are not isometric: \begin{align*}B(u,v) &= 19u_1v_1 - u_1v_2 - u_2v_1 + u_1v_2 + 18u_3v_3,\\ B'(u,v) &= 108u_1v_1 + 3u_2v_2 + u_3v_3,\\ B''(u,v) &= 60u_1v_1 + 6u_1v_2 + 6u_2v_1 + 6u_2v_2 + u_3v_3.\end{align*}
\end{eg}

\begin{eg}
\label{eg:rank4}
    Taking $n=4$, it will no longer be sufficient in general to use only the first equation in order to obtain a contradiction. For example, consider
    \begin{align*}
        B(u,v)&=2u_1v_1+u_1v_2+u_2v_1+2u_2v_2+2u_3v_3+4u_4v_4, \\
        B^\prime(u,v)&=2u_1v_1+u_1v_2+u_2v_1+u_1v_3+u_3v_1+2u_2v_2+2u_3v_3+6u_4v_4.
    \end{align*}
    We choose $w=(1,0,0,0)$, so that $B(w,w)=B^\prime(w,w)=2$ and $\tilde{b}=(b_1,-2b_1,b_3,b_4)$ for some $b_1,b_3,b_4\in\mathbb{Z}$. Equation one of Corollary \ref{cor:quadform} gives us the equation:
    \begin{equation*}
        8=6b_1^2+2b_3^2+4b_4^2+2B(w,\phi(w))^2.
    \end{equation*}
    There are 10 putative integer solutions to this equation. Next, we make the choices:
    \begin{equation*}
        z_{01}=(0,1,0,0), z_{02}=(0,0,1,0), z_{03}=(0,0,0,1).
    \end{equation*}
    for $z_0$ in Corollary \ref{cor:quadform} and seek solutions to equation three. We then substitute all putative solutions to equations one and three into equation two. In total we find that $309$ solutions remain, and for these we were then able to reconstruct the quantities $\tilde{a}$, $\tilde{b}$, $\tilde{\phi}_0$ and $B(w, \phi(w))$, hence recovering all possible rational isometries $\phi$. Doing this gives the following 8 matrices $M\in\text{GL}_4(\mathbb{Q})$ satisfying $\mathcal{B}' = M^T\mathcal{B}M$:
    \begin{align*}
     &\left(\begin{smallmatrix}  \frac{1}{2} & \frac{3}{4} & \frac{3}{4} & \frac{1}{2} \\
     -1 & -\frac{1}{2} & -\frac{1}{2} & -1 \\ -\frac{1}{2} & -\frac{1}{4} & -\frac{1}{4} & \frac{3}{2} \\
     0 & -\frac{1}{2} & \frac{1}{2} & 0 \end{smallmatrix} \right),  \left(\begin{smallmatrix}  \frac{1}{2} & \frac{3}{4} & \frac{3}{4} & \frac{1}{2} \\
     -1 & -\frac{1}{2} & -\frac{1}{2} & -1 \\ -\frac{1}{2} & -\frac{1}{4} & -\frac{1}{4} & \frac{3}{2} \\
     0 & \frac{1}{2} & -\frac{1}{2} & 0 \end{smallmatrix} \right),
     \left(\begin{smallmatrix}  \frac{1}{2} & \frac{3}{4} & \frac{3}{4} & \frac{1}{2} \\
     -1 & -\frac{1}{2} & -\frac{1}{2} & -1 \\ \frac{1}{2} & \frac{1}{4} & \frac{1}{4} & -\frac{3}{2} \\
     0 & -\frac{1}{2} & \frac{1}{2} & 0 \end{smallmatrix} \right)
     \left(\begin{smallmatrix}  \frac{1}{2} & \frac{3}{4} & \frac{3}{4} & \frac{1}{2} \\
     -1 & -\frac{1}{2} & -\frac{1}{2} & -1 \\ \frac{1}{2} & \frac{1}{4} & \frac{1}{4} & -\frac{3}{2} \\
     0 & \frac{1}{2} & -\frac{1}{2} & 0 \end{smallmatrix} \right), \\
     &\left(\begin{smallmatrix}  \frac{1}{2} &  1  & 0  & 1 \\
     0  & 0 & 0 & -2\\ \frac{1}{2} & 0 & 1 & 0 \\-\frac{1}{2} & 0 & 0 & 0 \end{smallmatrix} \right),
     \left(\begin{smallmatrix}  \frac{1}{2} &  1  & 0  & 1 \\
     0  & 0 & 0 & -2\\ \frac{1}{2} & 0 & 1 & 0 \\\frac{1}{2} & 0 & 0 & 0 \end{smallmatrix} \right),
     \left(\begin{smallmatrix} 0 & \frac{1}{2} & \frac{1}{2} & 1 \\
    0 & 0 & 0 & -2 \\ 1 & \frac{1}{2} & \frac{1}{2} & 0 \\  0 & -\frac{1}{2} & \frac{1}{2} & 0  \end{smallmatrix} \right),
    \left(\begin{smallmatrix} 0 & \frac{1}{2} & \frac{1}{2} & 1 \\
    0 & 0 & 0 & -2 \\ 1 & \frac{1}{2} & \frac{1}{2} & 0 \\  0 & \frac{1}{2} & -\frac{1}{2} & 0  \end{smallmatrix} \right).
    \end{align*}
    Since none of these matrices are in $GL_4(\mathbb{Z})$, we deduce there is no isometry between $(\Lambda,B)$ and $(\Lambda,B')$.
\end{eg}
In all the previous examples, we have been attempting to find a solution to $\mathcal{B}' = M^T \mathcal{B}M$ with $M,\mathcal{B},\mathcal{B}'\in \text{GL}_n(\mathbb{Z})$ and $\mathcal{B},\mathcal{B}'$ symmetric and not equal to the identity matrix. These examples were not accessible to the authors of \cite{higham}, \cite{hill}, \cite{lettington} as they only considered the case where $\mathcal{B}=I$, the identity matrix. In the example below, we will revisit an example studied in \cite{higham} to see a more direct comparison of the methods.
\begin{eg}
\label{eg:wilson}
    In Section 3 of \cite{higham}, the authors use the superalgebra structure corresponding to the choices $\mathcal{B} = I_4$ and $w=(1,1,1,1)$ (in our notation) to determine an integer matrix factorisation $W = M^TM$, where $W$ is the Wilson matrix
    \begin{equation*}
        W=\left(\begin{smallmatrix}
            5 & 7 & 6 & 5\\ 7 & 10 & 8 & 7\\ 6 & 8 & 10 & 9\\ 5 & 7 & 9 & 10
        \end{smallmatrix} \right).
    \end{equation*}
    This particular matrix was chosen by the authors to test their methods, due to $W$ being mildy ill conditioned. By considering an analogue of the first equation of Corollary \ref{cor:quadform}, they find that there are $1728$ integer solutions. Each is then studied in detail in order to figure out which pairs of vectors $a,b\in\{w\}^{\perp}$ lead to an integer matrix factorization (as opposed to a rational one). This calculation is reported to have taken a total $34$ minutes on Mathematica (more details can be found in the paper). 
    
    Using our more general methods, we were able to shorten this calculation to under a second, highlighting the advantages of our approach. In particular the choice $w=(1,0,0,0)$ leads to only $24$ solutions to the first equation of Corollary \ref{cor:quadform}. As in the example above, we then used the third and second equations of the corollary to recover all possible maps $\phi$. Doing this quickly revealed $96$ possible $M \in GL_4(\mathbb{Z})$ such that $W = M^TM$, all of the form $M=UM^\prime$ where $U$ is an orthogonal matrix with integer entries and
    \begin{equation*}
        M^\prime=\left(\begin{smallmatrix}
            2&3&2&2 \\ 1&1&2&1 \\ 0&0&1&2 \\ 0&0&1&1
        \end{smallmatrix}\right).
    \end{equation*}
    This matches the outcome in \cite{higham}, although allowing flexibility in the choice of $w$ has allowed for a more efficient computation. 
    \end{eg}

\end{document}